\documentclass[12pt]{amsart}
\usepackage{amsmath,amssymb,amsfonts,amsthm,amsopn,comment}
\usepackage{graphics}
  \def\Spnr{Sp(d,\R)}
  \def\Gltwonr{GL(2d,\R)}

\newcommand{\tfa}{time-frequency analysis}
\newcommand{\tfr}{time-frequency representation}
\newcommand{\ft}{Fourier transform}

\newcommand{\psdo}{pseudodifferential operator}

\newcommand{\tpsdo}{$\tau$-pseudodifferential operator}

\def\twd{$\tau$-WD}
\def\la{\lambda}

\def\t{\tau}

\def\cF{\mathcal{F}}              % Calligraphic Letters
\def\cS{\mathcal{S}}

\def\cI{\mathcal{I}}

\def\cR{\mathcal{R}}

\def\R{\mathbb{R}}
\def\Ren{\mathbb{R}^d}
\def\Renn{\mathbb{R}^{2d}}

\def\distr{\mathcal{S}'}
\def\sch{\mathcal{S}}  
\def\intrd{\int_{\Ren}}
\def\intrdd{\int_{\Renn}}

\def\Wt{W_{\tau}}

\def\Fur{\mathcal{F}}
\def\rd{\mathbb{R}^d}
\def\rdd{\mathbb{R}^{2d}}
\def\la{\langle}
\def\ra{\rangle}
\def\phas{(x,\xi )}
\def\f{\varphi}
\def\z{\zeta}
\def\Mmpq{M_m^{p,q}}
\def\o{\xi}
\def\og{\omega}
\def\A{\mathcal{A}_{\tau}}
\def\s{\mathcal{S}(\rd)}
\def\S{\mathcal{S}'(\rd)}
\def\Ft{\mathcal{F}}
\def\tauop{\operatorname*{Op}\nolimits_{\tau}(a)}
\def\tauopnosy{\operatorname*{Op}\nolimits_{\tau}}
\def\Opkn{\operatorname*{Op}\nolimits_{0}}
\def\Opknc{\operatorname*{Op}\nolimits_{1}}
\def\Opw{\operatorname*{Op}\nolimits_{1/2}}

\def\l{\langle}
\def\r{\rangle} 

\newtheorem{tm}{Theorem}[section]
\newtheorem{lemma}[tm]{Lemma}
\newtheorem{prop}[tm]{Proposition}
\newtheorem{cor}[tm]{Corollary}

\newtheorem{theorem}{Theorem}[section]
\newtheorem{corollary}[theorem]{Corollary}
\newtheorem{definition}[theorem]{Definition}

\newtheorem{proposition}[theorem]{Proposition}%%
\newtheorem{remark}[theorem]{Remark}

\begin{document}

%----------Author 1

\title[ Norm Estimates for $\tau$-Pseudodifferential Operators]{Norm Estimates for $\tau$-Pseudodifferential operators  in Wiener Amalgam and Modulation Spaces}
\address{Dipartimento di Matematica,
Universit\`a di Torino, via Carlo Alberto 10, 10123 Torino, Italy}
\author{Elena  Cordero}
\email{elena.cordero@unito.it}
\address{Dipartimento di Matematica,
Universit\`a di Torino, via Carlo Alberto 10, 10123 Torino, Italy}
\author{Lorenza D'Elia}
\email{lorenza.delia@unito.it}
\author{S. Ivan Trapasso}
\address{Dipartimento di Scienze Matematiche,
Politecnico di Torino, Corso Duca degli Abruzzi 24, 10129 Torino, Italy}
%\author{S. I. Trapasso}
%\email{fabio.nicola@polito.it}
%\email{maurice.de.gosson@univie.ac.at}
\email{salvatore.trapasso@polito.it}
%\thanks{This work was completed with the support of our
%\TeX-pert.}
%----------Author 2

%----------classification, keywords, date
%\thanks{{\bf Acknowledgments.} The authors thank Luigi Rodino  for very helpful  discussions on this
%topic.}
\date{}
%----------additions
%\dedicatory{To my parents}
%%% ----------------------------------------------------------------------

\begin{abstract}
We study continuity properties on modulation spaces for $\t$-pseudo--differential operators $\tauop$ with symbols $a$ in Wiener amalgam spaces. We obtain boundedness results for $\tau \in (0,1)$ whereas, in the end-points $\tau=0$ and $\tau=1$, the corresponding operators are in general unbounded. Furthermore, for $\tau \in (0,1)$, we exhibit a function of $\tau$  which is an upper bound for the operator norm. 
 The continuity properties of  $\tauop$, for any $\tau\in [0,1]$, with  symbols $a$ in modulation  spaces are well known. Here we find an upper bound for the operator norm which does not depend on the parameter $\tau \in [0,1]$, as expected.\par 
Key ingredients  are uniform continuity estimates for $\t$-Wigner distributions.
\end{abstract}

%%% ----------------------------------------------------------------------
\maketitle
%%% ----------------------------------------------------------------------
\section{Introduction}
Pseudodifferential operators are mathematical tools used extensively in the theory of partial differential equations, engineering  and quantum mechanics. Since  their first appearance in the works by Kohn, Nirenberg \cite{KN} and H\"ormander \cite{H}, they have been widely studied in the framework of classical analysis by plenty of authors, with privileged symbol classes being the so-called H\"ormander class  \cite{H}. In this context, we also refer to the textbooks \cite{Shubin2001,Stein93,Taylor81,Treves80}.

Starting from the end of the 90's and during the last 20 years they have been considered in the context of \tfa. Many outcomes have been obtained, showing in particular that operators with rough symbols (functions not even differentiable or tempered distributions) may be bounded on $L^2(\rd)$. The contributions are so many that we are not able to cite them all. See, for instance, \cite{BGHO,Boul,CTW,DT,GH,GH1,GS,Labate2,Ruzhansky2016,Sjo95,Tachizawa1,Nenad2018,Nenad2016,T,Toftweight2004}.

Assume $\tau\in [0,1]$, $a\in \cS'(\rdd)$, then the $\tau$-pseudodifferential operator $\tauop$ with symbol $a$ can be defined by
\begin{equation}\label{tauop}
\tauop f(x)=\intrd\intrd e^{2\pi i (x-y) \xi} a((1-\tau)x +\tau y,\xi) f(y) \,dyd\xi,\quad f\in\cS(\rd).
\end{equation}
If $\tau=0$, the corresponding operator $\Opkn(a)$ is called the Kohn-Niremberg operator and can be rewritten as
 \begin{equation}
 \label{KNoperator}
 \Opkn(a)f(x)= \intrd a(x,\xi)\hat{f}(\xi)e^{2\pi ix\xi} d\xi, \qquad f\in \mathcal{S}(\rd).
 \end{equation}
 For $\tau=1/2$,   $\Opw(a)$ (proposed by Weyl in \cite{W}) is called  Weyl operator and takes the form 
 \begin{equation}
 \label{Woperator}
 \Opw(a)f(x)= \intrdd a\biggl(\frac{x+y}{2},\xi\biggr)f(y)e^{2\pi i(x-y)\xi} dyd\xi, \qquad f\in \mathcal{S}(\rd).
 \end{equation}

In this paper we continue the study of boundedness properties of \psdo s using tools from time-frequency analysis.
 The main ingredients  are the  time-frequency representations related to the definition of $\t$-pseudodifferential operators. For  $\tau\in [0,1]$, the (cross-)$\tau$-Wigner distribution ($\tau$-WD) of signals  $f,g\in L^2(\Ren)$ is defined by 
\begin{equation}
\label{tau-Wigner distribution}
W_{\tau}(f,g)(x,\xi) = \intrd e^{-2\pi it\xi}f(x+\tau t)\overline{g(x-(1-\tau)t)}dt,\quad\phas\in\rdd.
\end{equation}
For $f=g$, $W_{\tau}f :=W_{\tau}(f,f)$ is called the $\tau$-Wigner distribution of $f$.
Note that  $\Wt f$ is a quadratic \tfr\ which is a generalization of the well known Wigner distribution, recaptured in the case  $\tau=1/2$:
 $$	W_{1/2}(f,g)(x,\xi)=W(f,g)(x,\xi)=\intrd f\biggl(x+\frac t2\biggr)\overline{g\biggl(x-\frac t2\biggr)}e^{-2\pi i t\xi}\,dt.$$
For $\tau =0$, $W_0(f,g)$ is named  (cross-)Rihaczek distribution 
\begin{equation}\label{Rihdistr}
W_{0}(f,g)(x,\xi)=\cR(f,g)(x,\xi)=e^{-2\pi ix\cdot\xi}f(x)\overline{\hat{g}(\xi)};
\end{equation} 
and for $\tau=1$, $W_{1}(f,g)$ is the (cross-)conjugate Rihaczek distribution 
\begin{equation}\label{coRihdist}
W_{1}(f,g)(x,\xi) = \cR^{*}(f,g)(x, \xi) = \overline{\cR(g,f)(x,\xi)}= e^{2\pi ix\cdot\xi}\overline{g(x)}\hat{f}(\xi).
\end{equation}
Given a symbol $a\in\mathcal{S}'(\rdd)$, the $\t$-pseudodifferential operator $\tauop$ in \eqref{tauop} can be defined weakly as a duality between the symbol $a$ and the \twd\, $W_\tau(g,f)$ as follows
 $$\l\tauop f,g\r = \l a, W_{\tau}(g,f)\r, \qquad f,g\in\cS(\Ren).$$

 Inspired by the work of Boulkhemair \cite{Boul}, we continue his investigation considering  symbols in the new framework of Wiener amalgam spaces. Such spaces can be viewed as $L^q(L^p)$-norm of a time-frequency representation: the  short-time Fourier transform (STFT) $V_g f$ of a signal $f\in\cS'(\rd)$ with respect to a window function $g\in\cS(\rd)$, defined by
\begin{equation}\label{STFT}
V_gf(z)=\langle f,\pi(z)g\rangle=\Fur [fT_x g](\xi)=\int_{\Ren}
f(t)\, {\overline {g(t-x)}} \, e^{-2\pi it \o }\,dt,
\end{equation}
for $z=(x,\o)\in\rd\times \rd$. For simplicity, we recall their definition in the unweighted case, referring to the next section for a more general definition and related properties. A tempered distribution $f\in\cS'(\rd)$ is in the Wiener amalgam space $W(\Fur L^p,L^q)(\rd)$, $1\leq p,q\leq\infty$,  if 
\[
\|f\|_{W(\Fur L^p,L^q)(\rd)}:=\left(\int_{\Ren}
\left(\int_{\Ren}|V_gf(x,\o)|^p \,
d\o\right)^{q/p} d x\right)^{1/q}<\infty . \,
\]
Roughly speaking,  a distribution $f$ is in the space $W(\Fur L^p,L^q)(\rd)$ if \emph{locally} it   behaves like a function in $\Fur L^p(\rd)$ and \emph{globally} decays as a function in $L^q(\rd)$. Such spaces capture the different behaviour of functions/distributions on local and global levels. For instance, it can be shown that the delta distribution $\delta$ is in  $W(\cF L^\infty, L^1)(\rd)$: its \ft\,$\cF \delta=1$ belongs to $L^\infty(\rd)$ and the compact support guarantees whatever decay at infinity.

Modulation spaces are closely related  to such Wiener spaces. Indeed  the  modulation space $M^{p,q}(\rd)$ can be defined by
$$M^{p,q}(\rd)=\cF^{-1} W(\cF L^p,L^q)(\rd),
$$
where $\cF^{-1}$ is the inverse \ft.

Sufficient and necessary conditions for boundedness properties of pseudodifferential operators with symbols in modulation spaces and acting on the same spaces have been found in many papers, cf. \cite{CN,CTW2013,T,Toftweight2004} and the bibliography therein. Here such conditions do not depend on the parameter $\tau\in [0,1]$, see Section $5$ for details in this framework.

In our context we continue to study boundedness properties on modulation spaces but the symbols are in the Wiener ones. Here the continuity properties do depend on the parameter $\tau$. 

%Elena

For  $1\leq r_1,r_2\leq\infty$,  we introduce the function
\begin{equation}\label{alfatau}
\alpha_{(r_1,r_2)}(\tau)=\frac{1}{\tau^{d\left(\frac{1}{r_1'}+\frac{1}{r_2}\right)}(1-\tau)^{d\left(\frac{1}{r_1}+\frac{1}{r_2'}\right)}},\quad  \tau\in (0,1).
\end{equation}
Observe that the function $\alpha_{(r_1,r_2)}(\tau)$ is unbounded on $(0,1)$. Indeed, for $(r_1,r_2)\not\in \{(1,\infty), (\infty,1)\}$, 
$$\lim_{\tau\to 0^+} \alpha_{(r_1,r_2)}(\tau)=\lim_{\tau\to 1^-} \alpha_{(r_1,r_2)}(\tau)=+\infty.$$
For $(r_1,r_2)=(1,\infty)$ we have $\lim_{\tau\to 1^-} \alpha_{(1,\infty)}(\tau)=+\infty$ whereas, for $(r_1,r_2)=(\infty,1)$,  $\lim_{\tau\to 0^+} \alpha_{(\infty,1)}(\tau)=+\infty$.
An unweighted version of our main result, cf. Theorem \ref{mainthm} below, can be read as follows: 
\begin{theorem}
Suppose that $1 \leq p,q,r_1,r_2 \leq \infty$ satisfy
\begin{equation*}
q\leq p',\quad \max \{r_1,r_2,r_1',r_2'\}\leq p.
\end{equation*}
Let  $a$ be a symbol in $ W(\cF L^p, L^q)(\rdd)$. For $\tau\in(0,1)$, every \tpsdo\ $\tauop$ is a bounded operator on $M^{r_1,r_2}(\R^d)$. Moreover, there exists  a constant $C>0$ independent of $\tau$ such that
\begin{equation}\label{stimae1}
\|\operatorname*{Op}\nolimits_{\tau}(a)f\|_{M^{r_1,r_2}} \leq C \alpha_{(r_1,r_2)}(\tau)
\|a\|_{W(\cF L^p,L^q)}\|f\|_{M^{r_1,r_2}},\quad  \tau \in (0,1).
\end{equation}
\end{theorem}
Hence we have found an upper bound for the operator norm:
$$\| \operatorname*{Op}\nolimits_{\tau}(a)\|_{\mathcal{B}(M^{r_1,r_2})}\leq C \alpha_{(r_1,r_2)}(\tau)
\|a\|_{W(\cF L^p,L^q)}. 
$$
The unboundedness of the function $\alpha_{(r_1,r_2)}(\tau)$ in the end-points suggests that the boundedness results above fail in the case of Kohn-Nirenberg operators $\Opkn(a)$ and of operators \emph{with right symbol} $\Opknc(a)$ (also called anti-Kohn-Nirenberg operators). 
Indeed, we exhibit precise counterexamples in Proposition \ref{controesempi} below.

The paper is organized as follows.  Section $2$ is focused on the preliminary definitions and properties of $\t$-Wigner distributions and the involved function spaces. In Section $3$, we study the continuity properties of $\Wt(f,g)$ in the Wiener amalgam spaces, obtaining uniform estimates with respect to the parameter $\tau$. Section $4$ is devoted to the proof of the main theorem: Theorem \ref{mainthm}. We also treat the cases $\tau=0$ and $\tau=1$, showing examples of unbounded operators.  Section $5$ provides some useful remarks on the continuity results of $\tauop$ with symbol in modulation spaces. \\

\textbf{Notation.} We define the scalar product on $\Ren$ by  $xy=x\cdot y$. The Schwartz class is denoted by  $\sch(\Ren)$ and its dual, the space of tempered
distributions, by  $\sch'(\Ren)$. The brackets $\langle \cdot, \cdot \rangle$ stand for the inner product on $L^2(\Ren)$ or for duality pairing between a tempered distribution in $\distr$ and a function in $\cS$ (for convention it is antilinear in the second argument).\\ We write $f\lesssim g$ to indicate $f(x)\leqslant Cg(x)$ for every $x$ and some constant $C$, and similarly for $\gtrsim$. The notation $f\asymp g$ stands for $f\lesssim g$ and $f\gtrsim g$. We use a normalized Fourier transform
\[
\Fur f(\xi)= \int_{\Ren} e^{-2\pi i x\xi} f(x)\, dx.
\]
The translation operator $T_x$ of a function $f$ on $\rd$ is defined as $T_xf(t)=f(t-x)$ and the modulation operator $M_{\o}f(t)=e^{2\pi i\xi t}f(t).$ For $z=\phas$, we denote the so-called time-frequency shift acting on a function or distribution as $\pi(z)f(t)= M_{\o}T_xf(t)$. The reflection operator is defined as $\cI f(x)=f(-x)$.
For $1\leq p\leq \infty$, the conjugate exponent $p'$ of $p$ is the one that satisfies $1/p +1/p'=1$.

\section{Time-frequency Representations and Function Spaces}
Denote by $J$ the canonical symplectic matrix in $\mathbb{R}^{2d}$:
\begin{equation}\label{Jsymplmatr}
J=\left(\begin{array}{cc}
0_{d\times d} & I_{d\times d}\\
-I_{d\times d} & 0_{d\times d}
\end{array}\right)\in {Sp}(2d,\mathbb{R}),
\end{equation}
where  the
symplectic group ${Sp}(2d,\mathbb{R})$ is defined
by
$$
\Spnr=\left\{M\in\Gltwonr:\;M^{\top}JM=J\right\}.
$$
In the sequel we shall heavily use the following symplectic matrix
 \begin{equation}
 \label{matrA}
 \A=
 \begin{pmatrix}
 0_{d\times d} & (\frac{1-\tau}{\tau})^{1/2}I_{d\times d} \\
 -(\frac{\tau}{1-\tau})^{1/2}I_{d\times d} & 0_{d\times d}
 \end{pmatrix},\quad \tau\in (0,1).
 \end{equation}
The main properties of  $\A$ are detailed below. Their proof is attained by easy computations. 
\begin{lemma}
	~\label{decmatrA} For any $\tau\in (0,1)$,
	the matrix $\A$ in \eqref{matrA} enjoys  the following properties:
	\begin{enumerate}
		\item[$(i)$] $\mathcal{A}_{\tau}\in {Sp}(d,\mathbb{R})$; in particular,
		$\mathcal{A}_{1/2}=J$.
		\item[$(ii)$] $\mathcal{A}_{\tau}^{\top}=-\mathcal{A}_{1-\tau}$, $\mathcal{A}_{\tau}^{-1}=-\mathcal{A}_{\tau}$.
		\item[$(iii)$] $\mathcal{A}_{1-\tau}\mathcal{A}_{\tau}=\mathcal{A}_{\tau}^{\top}\mathcal{A}_{\tau}^{-1}=I_{2d\times2d}-\mathcal{B}_{\tau}$,
		where 
		\begin{equation}
		\mathcal{B}_{\tau}=\left(\begin{array}{cc}
		\frac{1}{1-\tau}I_{d\times d} & 0_{d\times d}\\
		0_{d\times d} & \frac{1}{\tau}I_{d\times d}
		\end{array}\right).\label{eq:Btau}
		\end{equation}
		\item[$(iv)$] $\sqrt{\tau\left(1-\tau\right)}\left(\mathcal{A}_{\tau}+\mathcal{A}_{1-\tau}\right)=\sqrt{\tau\left(1-\tau\right)}\mathcal{B}_{\tau}\mathcal{A}_{\tau}=J$.
	\end{enumerate}
\end{lemma}
\subsection{$\tau$-Wigner Distributions and their Short-Time Fourier Transforms.} 
We list now some useful features enjoyed by the \twd \, which we will use later (cf. \cite{BDDO,dG}).
\begin{prop} For $\tau\in [0,1]$, $f,g,f_i,g_i\in L^2(\Ren)$, $i=1,2$, we have \\
    (i) $W_{1-\tau}(f,g) = \overline{W_\tau (g,f)}$.\\
 	(ii) $W_\tau f(x,\xi) = W_{1-\tau}\hat{f}(\xi,-x)$. Equivalently  $$ W_\tau \hat{f}(z)=W_{1-\tau}f(Jz),$$
 	where $J$ is the canonical symplectic matrix in \eqref{Jsymplmatr}.\\
 	(iii) Moyal's Formula for \twd:
 	 	\begin{equation}\label{C1MoyaltauWigner}
 	 	\la W_\tau (f_1,g_1), W_\tau (f_2,g_2)\ra =\la f_1,f_2\ra \overline{\la g_1,g_2\ra }.
 	 	\end{equation}
 	 	(iv) Covariance property for the \twd:
 	 	\begin{equation}\label{CovtWig}
 	 	W_\tau(\pi(w) f,\pi(w) g)(z)= T_w W_\tau(f,g)( z)=W_\tau(f,g)( z-w),\quad w,z\in\rdd.
 	 	\end{equation}
\end{prop} 
%Note that from (i) we deduce at once 
%\begin{equation*}
 %W_\tau f + W_{1-\tau}f = 2\Re (W_\tau f), \qquad W_\tau f - W_{1-\tau}f = 2\Im (W_\tau f).
 %\end{equation*}

To study continuity properties of the \twd\, on modulation and Wiener spaces, we need to compute its Short-time Fourier transform (STFT). Recall that the STFT of a signal $f\in\distr(\rd)$ with respect a fixed window function $g\in\sch(\rd)$ is defined in
\eqref{STFT}. Important properties of STFT we shall use are as follows.
\begin{proposition}
For $f,f_i,g, g_i\in L^2(\rd)$, $i=1,2$, we have:\\
       (i) Orthogonality relations for the STFT:
          		 	\begin{equation}\label{OR}
          		 	\la V_{g_1} f_1,  V_{g_2} f_2\ra_{L^2(\rdd)}=\la f_1,f_2\ra_{L^2(\rd)} \overline{\la g_1,g_2\ra}_{L^2(\rd)}.
          		 	\end{equation}
 		(ii) STFT of time-frequency shifts:
 			\begin{equation} \label{STFTtfshift}
 		V_{M_{\omega}T_ug}(M_{\omega}T_uf)(x,\xi) = e^{2\pi i(\og x-\xi u)}V_gf(x, \xi), \qquad u, x,\xi, \og \in \mathbb{R}^d.
 		\end{equation}
 		(iii)	For $g_0, g, \gamma \in\s$ such that $\langle \gamma, g\rangle \neq 0$, $f\in\S$,  
 			\begin{equation}
 			\label{chwindow}
 			|V_{g_0}f\phas|\leq \frac{1}{|\langle \gamma, g\rangle|}(|V_gf|\ast|V_{g_0}\gamma|)\phas, \qquad \phas\in\rdd.
 			\end{equation}
 			(iv) Fundamental identity of time-frequency analysis:
 			\begin{equation}\label{FI}
 			V_g f(x,\o)= e^{-2\pi i x\o}V_{\hat g} \hat f(\o,-x),\quad \phas\in\rdd.
 			\end{equation}
\end{proposition}
For $\tau\in (0,1)$, the $\t$-Wigner distribution can be rephrased as a STFT, the key ingredient is the operator $A_{\tau}$  below.
 \begin{definition}
 	For $\tau\in (0,1)$, we define the operator $A_{\tau}$ by
 	\begin{equation}\label{opAtau}
 	A_{\tau}:f(t)\longmapsto \cI {f}\biggl(\frac{1-\tau}{\tau}t\biggr).
 	\end{equation}
 \end{definition}
 Then  (cf. \cite[Lemma 6.2]{BDDO}):
  \begin{lemma} 
  	For $\tau\in (0,1)$, $f,g\in L^2(\rd)$, we have 
  	\begin{equation}\label{WtauSTFT}
  	W_\tau (f,g)(x,\xi)= \frac{1}{\tau^d}e^{2\pi i\frac{1}{\tau}x\xi}V_{A_{\tau}g}f\left(\frac{1}{1-\tau}x,\frac{1}{\tau}\xi\right),\quad \phas\in\rdd.
  	\end{equation}
  \end{lemma}
 The proof of the following lemma is a matter of computation. 
 \begin{lemma}\label{C1commpiAtau}
 	For $\tau\in\left(0,1\right)$, $z=(z_1,z_2)\in\rdd$, the operators $A_{\tau}$ and $\pi(z)$ commute as follows
 	\begin{equation}\label{e1}
 	\pi\left(z\right)A_{\tau}=A_{\tau}\pi\left(-\frac{1-\tau}{\tau}z_1,-\frac{\tau}{1-\tau}z_2\right),
 	\end{equation}
 	\begin{equation}\label{e2}
 	A_{\tau}\pi\left(z\right)=\pi\left(-\frac{\tau}{1-\tau}z_1,-\frac{1-\tau}{\tau} z_2\right)A_{\tau}.
 	\end{equation}
 \end{lemma} 
In the next lemmas we calculate the STFT of $\Wt(g,f)$,  generalizing \cite[Lemma 4.3.1]{G}.
  	\begin{lemma}\label{STFT-tau-wigner} Consider $\tau\in (0,1)$. Let  $\f_1,\f_2\in\cS(\rd)$,
  		$f,g \in\cS(\rd)$ and set $\Phi_\tau= W_\tau(\f_1,\f_2)$. Then,   
  		\begin{equation}\label{STFTtauWig}
  		V_{\Phi_\tau} W_\tau(g,f)(z,\zeta)=e^{-2\pi i z_2\zeta_2}V_{\f_1} g(z_1-\tau \zeta_2, z_2+(1-\tau)\zeta_1)\overline{V_{\f_2} f(z_1+(1-\tau)\zeta_2,z_2-\tau \zeta_1)}
  		\end{equation}
  			where $z=(z_1,z_2), \zeta=(\zeta_1,\zeta_2)\in\rdd$.
  		\end{lemma}
 \begin{proof}
 Using the covariance property \eqref{CovtWig} and the representation of the $\tau$-Wigner distribution as a STFT in \eqref{WtauSTFT},
 \begin{align*}
 V_{\Phi_\tau} &W_\tau(g,f)(z,\zeta)= \la W_\tau(g,f), M_\zeta T_z W_\tau (\f_1,\f_2)\ra\\&= \la W_\tau(g,f), M_\zeta W_\tau(\pi(z)\f_1,\pi(z)\f_2)\ra \\
 &=\frac{1}{\tau^{2d}} \intrdd
  V_{A_\tau f}  g\left(\frac{1}{1-\tau}x,\frac{1}{\tau}\xi\right) e^{-2\pi i (x,\xi)\cdot (\zeta_1,\zeta_2)}\overline{V_{A_\tau \pi(z)\f_2}\pi(z)\f_1\left(\frac{1}{1-\tau}x,\frac{1}{\tau}\xi\right)}\,dx d\xi\\
 &= \frac{(1-\tau)^d}{\tau^d}\intrdd  V_{A_\tau f}  g\phas e^{-2\pi i (\zeta_1,\zeta_2)\cdot((1-\tau)x,\tau \xi)}\overline{V_{A_\tau \pi(z)\f_2}\pi(z)\f_1\phas }\,dx d\xi.
 \end{align*}
 To shorten notation, we write
 \begin{equation*}
c_\tau=\frac{(1-\tau)^d}{\tau^d}.
 \end{equation*}
 Using formula \eqref{STFTtfshift}, the orthogonality relations \eqref{OR} and the commutation relations between $\pi$ and $A_\tau$ in Lemma \ref{C1commpiAtau}, we compute
 \begin{align*}
 &V_{\Phi_\tau}  W_\tau(g,f)(z,\zeta)\\
 &\hspace{0.1truecm}=c_\tau\!\!\!\intrdd  V_{\pi(\tau \zeta_2,-(1-\tau) \zeta_1)A_\tau f}  \pi(\tau \zeta_2,-(1-\tau) \zeta_1)g
 \overline{V_{A_\tau \pi(z)\f_2}\pi(z)\f_1 }\phas dx d\xi\\
 &\hspace{0.1truecm}= c_\tau\la  \pi(\tau \zeta_2,-(1-\tau) \zeta_1)g,\pi(z_1,z_2)\f_1\ra \overline{\la \pi(\tau \zeta_2,-(1-\tau) \zeta_1)A_\tau f, A_\tau \pi(z_1,z_2)\f_2\ra }\\
 &\hspace{0.1truecm}=c_\tau e^{-2\pi i\tau  z_2\zeta_2}\la  g,\pi(z_1-\tau\zeta_2,z_2+(1-\tau)\zeta_1)\f_1\ra \\
 &\hspace{3truecm}\times \overline{\la A_\tau f, \pi(-\tau \zeta_2,(1-\tau) \zeta_1)A_\tau \pi(z_1,z_2)\f_2\ra }\\
 &\hspace{0.1truecm}=c_\tau e^{-2\pi i\tau  z_2\zeta_2}\la  g,\pi(z_1-\tau\zeta_2,z_2+(1-\tau)\zeta_1)\f_1\ra \\
 &\hspace{3truecm}\times \overline{\la A_\tau f, A_\tau \pi((1-\tau)\zeta_2, -\tau \zeta_1) \pi(z_1,z_2)\f_2\ra }\\
 &=e^{-2\pi i\tau  z_2\zeta_2}e^{-2\pi i (1-\tau)z_2\zeta_2}\la  g,\pi(z_1-\tau\zeta_2,z_2+(1-\tau)\zeta_1)\f_1\ra\\
 &\hspace{3truecm}\times \overline{\la  f,   \pi(z_1+(1-\tau)\zeta_2,z_2-\tau \zeta_1)\f_2\ra }\\
 &=e^{-2\pi i z_2\zeta_2}V_{\f_1} g(z_1-\tau\zeta_2,z_2+(1-\tau)\zeta_1)\overline{V_{\f_2} f(z_1+(1-\tau)\zeta_2,z_2-\tau \zeta_1)}.
 \end{align*}
 The claim is proved.
 \end{proof}

Formula \eqref{STFTtauWig} can be equivalently written as  
\begin{equation}\label{sympSTFT}
V_{\Phi_\tau} W_\tau(g,f)(z,\zeta)=e^{-2\pi i z_2\zeta_2} V_{\f_1} g(z+\sqrt{\tau(1-\tau)}\A^T\zeta) \overline{V_{\f_2} f(z+\sqrt{\tau(1-\tau)}\A\zeta)},
\end{equation}
 where $\mathcal{A}_{\tau}$ is a symplectic matrix defined in \eqref{matrA}.

The previous lemma does not cover the case $\t=0$ and $\t=1$, which are treated below.
 
\begin{lemma}[\emph{STFT of the Rihaczek distribution}] \label{STFT-Rihaczek} Let  $\f_1,\f_2\in\cS(\rd)$,
 	$f,g \in\cS(\rd)$ and set $\Phi_0= W_0(\f_1,\f_2)$. Then,   
 	\begin{equation}\label{STFTtau0}
 	V_{\Phi_0} W_0(g,f)(z,\zeta)=e^{-2\pi i z_2\zeta_2}V_{\f_1} g(z_1, z_2+\zeta_1) \overline{V_{\f_2} f(z_1+\zeta_2,z_2)},
 	\end{equation}
 	where $z=(z_1,z_2), \zeta=(\zeta_1,\zeta_2)\in\rdd$.
 \end{lemma}
 \begin{proof}
 	We use the definition in  \eqref{Rihdistr} and formula \eqref{STFTtfshift} in the following computations:
 	\begin{align*}
 	V_{\Phi_0} &W_0(g,f)(z,\zeta)= \la W_0(g,f), M_\zeta T_z W_0 (\f_1,\f_2)\ra\\
 	&= \intrdd
 	e^{-2\pi i x\xi} g(x)\overline{\hat{f}(\xi)} e^{-2\pi i (x \zeta_1+\xi \zeta_2)} e^{2\pi i(x-z_1)(\xi -z_2)}\overline{\f_1 (x-z_1)}\widehat{\f_2} (\xi-z_2) \,dx d\xi\\
 	&= e^{2\pi i z_1 z_2}\intrd e^{-2\pi i x(z_2+\zeta_1)}g(x)\overline{\f_1(x-z_1)}\,dx \intrd \overline{\hat{f}(\xi)} e^{-2\pi i \xi (z_1+\zeta_2)}\widehat{\f_2}(\xi - z_2)\,d\xi\\
 	&= e^{2\pi i z_1 z_2} V_{\f _1}g(z_1,z_2+\zeta_1)\overline{V_{\widehat{\f _2}}\hat{f}(z_2,-(z_1+\zeta_2))}\\
 	&=e^{2\pi i z_1 z_2} V_{\f _1}g(z_1,z_2+\zeta_1)\overline{V_{{\f _2}}{f}(z_1+\zeta_2,z_2)} e^{-2\pi i z_2(z_1+\zeta_2)}\\
 	&=e^{-2\pi i z_2 \zeta_2}V_{\f _1}g(z_1,z_2+\zeta_1)\overline{V_{{\f _2}}{f}(z_1+\zeta_2,z_2)},
 	\end{align*}
 	as desired.
 \end{proof}
 
 \begin{cor}[\emph{STFT of the conjugate-Rihaczek distribution}]   \label{STFT-Rihaczekconj} Let  $\f_1,\f_2\in\cS(\rd)$,
 	$f,g \in\cS(\rd)$ and set $\Phi_1= W_1(\f_1,\f_2)$. Then,   
 	\begin{equation}\label{STFTtau1}
 	V_{\Phi_1} W_1(g,f)(z,\zeta)=e^{-2\pi i z_2\zeta_2}V_{\f _1} g(z_1-\zeta_2,z_2)\overline{V_{\f _2}f(z_1,z_2-\zeta_1)},
 	\end{equation}
 	where $z=(z_1,z_2), \zeta=(\zeta_1,\zeta_2)\in\rdd$.
 	\end{cor}
 \begin{proof}
 	Using the connection between the Rihaczek and the conjugate-Rihaczek distribution in \eqref{coRihdist} and the result of Lemma \ref{STFT-Rihaczek} we can write
 	\begin{align*}
 	V_{\Phi_1} W_1(g,f)(z,\zeta)&=\la W_1(g,f), M_\zeta T_zW_1(\f _1,\f _2)\ra\\
 	&=\la \overline{W_0(f,g)},  M_\zeta T_z \overline{W_0(\f _2,\f _1)}\ra\\
 	&=\overline{\la W_0(f,g),M_{-\zeta} T_z W_0(\f _2,\f _1)\ra }\\
 	&=\overline{V_{W_0(\f_2,\f_ 1)}W_0(f,g)(z,-\zeta)}\\
 	&=\overline{e^{2\pi i z_2\zeta_2} V_{\f_2} f(z_1,z_2-\zeta_1)\overline{V_{\f_1} g(z_1-\zeta_2, z_2)}}\\
 	&=e^{-2\pi i z_2\zeta_2}V_{\f _1} g(z_1-\zeta_2,z_2)\overline{V_{\f _2}f(z_1,z_2-\zeta_1)}.
 	\end{align*}
 	The proof is completed.
 \end{proof}

 \begin{remark} (i) Heuristically, formulae \eqref{STFTtau0} and  \eqref{STFTtau1} can be inferred by putting $\tau=0$ and $\tau=1$ respectively in the expression \eqref{STFTtauWig}. \\
 	(ii) The STFT of a multilinear version of the Rihaczek distribution was computed in \cite[Lemma 3.3]{BGHO}, cf. formula $(3.3)$. However, there is a flaw in the phase factor of that formula. Indeed, the exponential $e^{2\pi i u_0 \cdot(u_1+\cdots+u_m)}$ should be replaced by $e^{2\pi i \sum_{i=1}^m u_i \cdot v_i}$, as the linear case $m=1$ in \eqref{STFTtau0} shows.
 	\end{remark}

\subsection{Generalized Gaussian Functions} In order to compute the norm of $\Wt$ in Wiener amalgam spaces,  generalized Gaussian functions will play a crucial role.\\
Given $a,b,c>0$, the generalized Gaussian function is defined as 
\begin{equation}\label{GG}
 f_{a,b,c}(x,\xi) =e^{-\pi ax^2}e^{-\pi b\xi^2}e^{2\pi icx\xi},\quad \phas\in\rdd.
\end{equation}
In the sequel, we will employ the STFT of a generalized Gaussian function, computed in  \cite[Proposition 2.2]{CN}:
\begin{proposition}
For $\Phi(x,\xi)=e^{-\pi(x^2+\xi^2)}$,  $z=(z_1,z_2)$, $\zeta=(\zeta_1,\zeta_2)\in\rdd$, we obtain
 	\begin{align}
 	V_\Phi f_{a,b,c}(z,\zeta)&= C(a,b,c)
 	e^{-\pi\frac{[a(b+1)+c^2]z_1^2+[(a+1)b+c^2]z_2^2+(b+1)\zeta_1^2+(a+1)\zeta^2_2-2c(z_1\zeta_2+z_2\zeta_1)}{(a+1)(b+1)+c^2}} \nonumber\\
 	&\quad\quad\quad\times\,\,  e^{-\frac{2\pi i}{a+1} \big[z_1\zeta_1+(cz_1-(a+1)\zeta_2)\frac{c\zeta_1+(a+1)z_2}{(a+1)(b+1)+c^2}\big]},
 	\label{STFTGauss}
 	\end{align}
 	with $C(a,b,c)=[(a+1)(b+1)+c^2]^{- d/2}$.
 \end{proposition}
The $\tau$-Wigner distribution of the Gaussian function $\f(t)=e^{-\pi t^2}$ is in turn a generalized Gaussian function, as showed in the next lemma.
\begin{lemma}
\label{tWDgaussian}
Consider $\varphi_1(t)=\varphi_2(t)=\f(t)=e^{-\pi t^2}$, $t\in\rd$, and $\tau\in [0,1]$. Then
\begin{equation*}
\label{WigGauss}
W_{\t}\f\phas = \frac{1}{(2\t^2-2\t+1)^{d/2}} e^{-\pi\tfrac{1}{2\t^2-2\t+1}x^2}e^{-\pi\tfrac{1}{2\t^2-2\t+1}\xi^2}e^{2\pi i\tfrac{2\t-1}{2\t^2-2\t+1}x\xi},
 \end{equation*}
for all $\phas\in\rdd$.
\end{lemma}
\begin{proof}
Using the definition of the \twd\,in \eqref{tau-Wigner distribution},
\begin{align*}
W_{\t}\varphi(x,\xi) &= \intrd e^{-2\pi i\xi t} e^{-\pi(x+\t t)^2}e^{-\pi(x-(1-\t)t)^2}dt\\
&= \intrd e^{-2\pi i\xi t}e^{-2\pi x^2}e^{-\pi[(2\t^2-2\t+1)t^2+2(2\t-1)xt]}dt \\
&=e^{-2\pi x^2 +\pi(\tfrac{2\t-1}{\sqrt{2\t^2-2\t+1}})^2x^2}\intrd e^{-2\pi i\xi t}e^{-\pi (\sqrt{2\t^2-2\t+1}t+\tfrac{2\t-1}{\sqrt{2\t^2-2\t+1}}x)^2} dt.
\end{align*}
We perform the following change of variables
\begin{equation*}
\sqrt{2\t^2-2\t+1}t+\frac{2\t-1}{\sqrt{2\t^2-2\t+1}}x = y,
\end{equation*}
so that, naming $c(\tau)= 2\t^2-2\t+1>0$,
\begin{align*}
W_{\t}\varphi(x,\xi) &= \frac{1}{c(\tau)^\frac d2}
	e^{-\pi \left(2-\tfrac{(2\t-1)^2}{c(\tau)}\right)x^2}\intrd e^{-2\pi i \xi \tfrac{y}{\sqrt{c(\tau)}}} e^{2\pi i\tfrac{2\t-1}{c(\tau)}\xi x} e^{-\pi y^2}dy\\ 
&=\frac{1}{c(\tau)^\frac d2}e^{-\pi \tfrac{1}{c(\tau)}x^2} e^{2\pi i\tfrac{2\t-1}{c(\tau)}\xi x} e^{-\pi \tfrac{1}{c(\tau)}\xi^2},
\end{align*}
as desired.
\end{proof}
 
\subsection{Weights and Function Spaces}
In time-frequency analysis, weight functions play an important role, since they describe the growth and the decay of a signal $f$ on the time-frequency plane $\rdd$. For a complete survey on weights, we refer to \cite{G1}. A weight function is a positive, locally integrable function on $\rdd$. In the sequel, we will need the following types of weight functions.
\begin{definition}
Let $v$ and $m$ be positive functions on $\rdd$.
\begin{enumerate}
\item[$(i)$ ] A weight $v$ is called submultiplicative if
\begin{equation*}
v(z_1+z_2)\leq v(z_1)v(z_2), \qquad \forall z_1,z_2\in\rdd.
\end{equation*}
\item[$(ii)$ ] Let $v$ be a submultiplicative weight, a positive function $m$ on $\rdd$ is called a $v$-moderate weight, if there exists a constant $C>0$, such that
\begin{equation*}
m(z_1+z_2)\leq Cv(z_1)m(z_2), \qquad \forall z_1,z_2\in\rdd.
\end{equation*}
\end{enumerate}
\end{definition}
Let $\mathcal{M}_{v}(\rdd)$ be the space of all $v$-moderate weights.
An important feature of  submultiplicative weights is that they have at most an exponential growth,  cf. \cite[Lemma 4.2]{G1}):
\begin{lemma}
\label{Lemgrowei}
If $v$ is submultiplicative and even weight, then there exist  constants $C,a>0$ such that
\begin{equation*}
v(x)\leq Ce^{a|x|}, \qquad \forall x\in\rdd.
\end{equation*}
\end{lemma}

From now on, we assume that $v$ is a continuous, positive, even, submultiplicative weight, i.e., $v(0)=1$, $v(z)=v(-z)$ and $v(z_1+z_2)\leq v(z_1)v(z_2)$, for all $z_1,z_2\in\rdd$.  In what follows, we will use weight $1/v$, which is a $v$-moderate weight:
\begin{equation*}
v(x)=v(x+y-y)\leq v(x+y)v(y)\quad\Rightarrow\quad\frac{1}{v(x+y)}\leq v(y)\frac{1}{v(x)}.
\end{equation*}
Weight functions occur in the definition of  general modulation spaces and Wiener amalgam spaces, where they offer a good device to measure a joint time-frequency concentration of a function or distribution. The definition of these function spaces relies on imposing a suitable norm on the short-time Fourier transform, defined in \eqref{STFT}. For their basic properties  we refer to  \cite{F1,F2,F3} and the textbooks \cite{dG,G}.\\Given a non-zero window $g\in\sch(\Ren)$, a $v$-moderate weight
function $m$ on $\Renn$, $1\leq p,q\leq
\infty$, the {\it
  modulation space} $M^{p,q}_m(\Ren)$ consists of all tempered
distributions $f\in\sch'(\Ren)$ such that $V_gf\in L^{p,q}_m(\Renn )$
(weighted mixed-norm spaces). The norm on $M^{p,q}_m(\Ren )$ is defined by
$$
\|f\|_{M^{p,q}_m}=\|V_gf\|_{L^{p,q}_m}=\left(\int_{\Ren}
  \left(\int_{\Ren}|V_gf(x,\o)|^pm(x,\o)^p\,
    dx\right)^{q/p}d\o\right)^{1/q}  \,
$$
(obvious modifications for $p=\infty$ or $q=\infty$). If $p=q$, we write $M^p_m(\Ren )$ instead of $M^{p,p}_m(\Ren )$, and if $m(z)\equiv 1$ on $\Renn$, then we write $M^{p,q}(\Ren )$ and $M^p(\Ren )$ for $M^{p,q}_m(\Ren )$ and $M^{p,p}_m(\Ren )$.

The space  $\Mmpq (\Ren )$ is a Banach space
whose definition is independent of the choice of the window $g$, in the sense that different  non-zero window functions yield equivalent  norms.
The modulation space $M^{\infty,1}(\rdd)$ is also called the Sj\"ostrand's class \cite{S}. We recall the inclusion properties of modulation spaces.
Suppose $m_1,m_2$ weight functions with $m_2 \lesssim
m_1$.  Then, for $1\leq p_1,p_2,q_1,q_2\leq \infty$,  with $p_1 \leq p_2$,  $ q_1 \leq q_2$,
\begin{equation}\label{modspaceincl1}
\cS (\rd) \subseteq M_{m_1}^{p_1,q_1} (\rd) \subseteq M_{m_2}^{p_2,q_2} (\rd) \subseteq
\sch'(\rd).
\end{equation} Note that for any $p,q \in [1,\infty]$ and any $m \in \mathcal{M}_{v}(\rdd)$,
the inner product $\la \cdot,\cdot \ra$ on $\cS (\rd) \times \cS (\rd)$ extends to a continuous sesquilinear map
$M^{p,q}_m (\rd) \times M^{p',q'}_{1/m} (\rd) \rightarrow \mathbb C$.\\
Given even weigh functions $u,w$ on $\rd$, the Wiener amalgam space $W(\Fur L^p_u,L^q_w)(\rd)$ consist of all distributions $f\in\cS'(\rd)$ such that
\[
\|f\|_{W(\Fur L^p_u,L^q_w)(\rd)}:=\left(\int_{\Ren}
  \left(\int_{\Ren}|V_gf(x,\o)|^p u^p(\o)\,
    d\o\right)^{q/p} w^q(x)d x\right)^{1/q}<\infty  \,
\]
where if $p=\infty$ or $q=\infty$, then we use the supremum norm.\\ The Wiener amalgam spaces $W(\Fur L^p_u,L^q_w)(\rd)$ are the image of modulation spaces $M^{p,q}_m(\Ren)$ under the Fourier transform
\begin{equation}\label{W-M}
\cF ({M}^{p,q}_{u\otimes w})(\rd)=W(\cF L_u^p,L_w^q)(\rd).
\end{equation}
Indeed, using Parseval identity in \eqref{STFT} and the fundamental identity \eqref{FI}, we can write $|V_g f(x,\o)|=|V_{\hat g} \hat f(\o,-x)| = |\mathcal F (\hat f \, T_\o \overline{\hat g}) (-x)|$  and (recall $u(x)=u(-x)$)
 $$
\| f \|_{{M}^{p,q}_{u\otimes w}} = \left( \int_{\rd} \| \hat f \ T_{\o} \overline{\hat g} \|_{\cF L^p_u}^q w^q(\o) \ d \o \right)^{1/q}
= \| \hat f \|_{W(\cF L_u^p,L_w^q)}.
$$
Hence  Wiener amalgam spaces are Banach spaces and their definition is independent of the choice of  $g$.\\ 
Modulation and Wiener amalgam space norms of signals are  weighted mixed-norm spaces of their short-time Fourier transforms. Hence their properties are based on those of the spaces $L^{p, q}_m$. Let us recall the convolution product of mixed-norm spaces \cite{BP61}:
\begin{lemma}\label{Younglemma}
	For $1\leq p_i,q_i,r,s\leq \infty$, $i=1,2$, $m\in \mathcal{M}_v(\rdd)$, $F\in L^{p_1,q_1}_{v}(\rdd)$, $G\in L^{p_2,q_2}_{m}(\rdd)$,  we have $F\ast G\in L_m^{r,s}(\rdd)$, with $1/p_1 +1/p_2=1+1/r$, $1/q_1+1/q_2=1+1/s$ and
	\begin{equation} \label{Youngineq}
	\| F\ast G\|_{L_m^{r,s}}\leq  \|F\|_{L^{p_1,q_1}_{v}}\|G\|_{L^{p_2,q_2}_{m}}.
	\end{equation}
\end{lemma}
We say that a measurable function $f$ on $\R^{4d}$ is in the space $L^{\infty}_z(L^1_{\zeta, m})(\R^{4d})$, with $m$ weight function on $\rdd$, if 
\begin{equation}\label{l1inf}
\|f\|_{L^{\infty}_z(L^1_{\zeta, m})}=\sup_{z\in\rdd}\int_{\rdd}|f(z,\zeta)|m(\zeta) d\zeta<\infty.
\end{equation}
When working on the STFT of $\tau$-WD,  we will use the following Young-type inequality:
\begin{lemma}
\label{propconvL1}
If $m\in\mathcal{M}_v(\rdd)$, $f\in L^1_{1\otimes v}(\R^{4d})$ and $g\in L^{\infty}_z(L^1_{\zeta, m})(\R^{4d})$, then $f\ast g \in L^{\infty}_z(L^1_{\zeta, m})(\R^{4d})$, with
\begin{equation*}
\|f\ast g\|_{L^{\infty}_z(L^1_{\zeta, m})} \leq \|f\|_{L^1_{1\otimes v}}\|g\|_{L^{\infty}_z(L^1_{\zeta, m})}.
\end{equation*}
\end{lemma}
\begin{proof}
Using the definition of $L^{\infty}_z(L^1_{\zeta, m})$-norm in \eqref{l1inf},
\begin{align*}
I:=\|f\ast g\|_{L^{\infty}_z(L^1_{\zeta, m})} &= \sup_{z\in\rdd}\intrdd |f\ast g|(z,\zeta)m(\zeta)d\zeta \\
&= \sup_{z\in\rdd}\intrdd \biggl|\int_{\R^{4d}} f(y,\eta)g(z-y, \zeta-\eta)dyd\eta\biggr| m(\zeta)d\zeta\\
&\leq \sup_{z\in\rdd}\intrdd \intrdd \biggl(\intrdd |f|(y,\eta)|g|(z-y, \zeta-\eta)d\eta\biggr)m(\zeta)dyd\zeta\\
&= \sup_{z\in\rdd}\intrdd \intrdd ( |f|(y,\cdot)\ast |g|(z-y, \cdot))(\zeta)m(\zeta)dyd\zeta.
\end{align*}
By Young's inequality \eqref{Youngineq}, 
\begin{align*}
I &= \sup_{z\in\rdd} \intrdd \||f|(y,\cdot) \|_{L^1_{v}}\| |g|(z-y,\cdot)\|_{L^1_m} dy\\ 
&\leq   \intrdd \||f|(y,\cdot) \|_{L^1_{v}} \sup_{z\in\rd}\| |g|(z-y,\cdot)\|_{L^1_m} dy\\
&= \|g\|_{L^{\infty}_z(L^1_{\zeta, m})} \|f\|_{L^1_{1\otimes v}},
\end{align*}
as claimed.
\end{proof}

A particular case of Lemma \ref{Younglemma} gives:
\begin{lemma}
\label{propconL2}
Suppose $m\in\mathcal{M}_v(\rdd)$,  $f\in L^1_{1\otimes v}(\R^{4d})$ and $g\in L^2_{1\otimes m}(\R^{4d})$. Then $f\ast g \in L^2_{1\otimes m}(\R^{4d})$, with
\begin{equation*}
\|f\ast g\|_{ L^2_{1\otimes m}} \leq \|f\|_{L^1_{1\otimes v}}\|g\|_{ L^2_{1\otimes m}}.
\end{equation*}
\end{lemma}

\section{Boundedness  properties  of $\tau$-Wigner Distributions}
This section is devoted to investigate the continuity properties of $\tau$-Wigner distributions in the realm of Wiener and modulation spaces.
For a submultiplicative weight $v$, we set  \begin{equation}\label{vJ}
v_J(z)=v(Jz),
\end{equation}where $J$ denotes the canonical symplectic matrix  \eqref{Jsymplmatr}. 
\begin{lemma}\label{Wigwienerlemma} Assume that $m \in \mathcal{M}_{v}(\rdd)$, $1\leq p_1,p_2\leq \infty$,  $f\in M^{p_1,p_2}_m(\rd)$, $g\in M^{p_1',p'_2}_{1/m}(\rd)$. Then  for every $\tau\in (0,1)$, the $\t$-Wigner distribution $W_\tau(g,f)$ is in $W(\cF L^1_{1/v_J}, L^\infty)(\rdd)$ , with
\begin{equation}\label{Wigwiener}
\|W_\tau(g,f)\|_{W(\cF L^1_{1/v_J}, L^\infty)} \leq C \alpha_{(p_1,p_2)}(\tau) \|f\|_{M^{p_1,p_2}_m}\|g\|_{M^{p_1',p'_2}_{1/m}},
\end{equation}
where the function $\alpha_{(p_1,p_2)}(\tau)$ is defined in \eqref{alfatau} and $C>0$ is independent of $\tau$.
\end{lemma}
\begin{proof}
We compute the STFT of $W_\tau(g,f)$ with respect to the window function $\Phi_\tau\in \cS(\rdd)$ defined in Lemma \ref{STFT-tau-wigner}. Using that lemma and the properties of the matrix  $\A$ in Lemma \ref{decmatrA}, by performing the change of variables $\sqrt{\tau(1-\tau)}\A\zeta = \eta$, we deduce
\begin{align*}
&\intrdd |V_{\Phi_\tau} W_\tau(g,f)|(z,\zeta)\frac{1}{v(J\zeta)}d\zeta  \\
&=\intrdd |V_{\f_1} g(z+\sqrt{\tau(1-\tau)}\A^T\zeta)||V_{\f_2} f(z+\sqrt{\tau(1-\tau)}\A\zeta)|\frac{1}{v(\sqrt{\tau(1-\tau)}\mathcal{B}_{\tau}\A\zeta)}d\zeta \\
&= \frac{1}{[\tau(1-\tau)]^d}\intrdd |V_{\f_1} g(z+\mathcal{A}_{1-\tau}\A\eta)||V_{\f_2} f(z+\eta)|\frac{1}{v(\mathcal{B}_\tau\eta)}d\eta.
\end{align*}
Since $m$ is a $v$-moderate weight, we can find a positive constant $C$, independent of $\t$, such that 
\begin{equation}
\label{inequweight}
\frac{1}{v(\mathcal{B}_\tau\eta)} \leq C \frac{m(z+\eta)}{m(z+\mathcal{A}_{1-\tau}\A\eta)},
\end{equation}
so that
\begin{align*}
&\intrdd |V_{\Phi_\tau} W_\tau(g,f)|(z,\zeta)\frac{1}{v(J\zeta)}d\zeta  \\
&\qquad\leq C\frac{1}{[\tau(1-\tau)]^d}\intrdd |V_{\f_1} g(z+\mathcal{A}_{1-\tau}\A\eta)||V_{\f_2} f(z+\eta)|\frac{m(z+\eta)}{m(z+\mathcal{A}_{1-\tau}\A\eta)}d\eta.
\end{align*}
Consequently,
\begin{align*}\label{star}
  &\|W_\tau(g,f)\|_{W(\cF L^1_{1/v_J}, L^\infty)} \\
  &\asymp  \sup_{z\in\rdd} \intrdd   |V_{\f_1} g(z+\sqrt{\tau(1-\tau)}\A^T\zeta)||V_{\f_2} f(z+\sqrt{\tau(1-\tau)}\A\zeta)|\frac{1}{v(\sqrt{\tau(1-\tau)}\mathcal{B}_{\tau}\A\zeta)}d\zeta \\
&\leq C\frac{1}{[\tau(1-\tau)]^d} \sup_{z\in\rdd} \intrdd |V_{\f_1} g(z+\mathcal{A}_{1-\tau}\A\eta)||V_{\f_2} f(z+\eta)|\frac{m(z+\eta)}{m(z+\mathcal{A}_{1-\tau}\A\eta)}d\eta\\
    &\leq C\frac{1}{[\tau(1-\tau)]^d} \|V_{\f_1} f m \|_{L^{p_1,p_2}} \|V_{\f_2} g \frac1m(z+\mathcal{A}_{1-\tau}\A\cdot)\|_{L^{p'_1,p'_2}}\\
    &\lesssim \frac{1}{[\tau(1-\tau)]^d}\biggl(\frac{1-\tau}{\tau}\biggr)^{d\left(\tfrac{1}{p_2}-\tfrac{1}{p_1}\right)}\|f\|_{M^{p_1,p_2}_m}\|g\|_{M^{p_1',p'_2}_{1/m}}.
  \end{align*}
  The claim is proved.
\end{proof}

The previous estimate is not uniform with respect to $\tau$, in the sense that the $W(\cF L^1_{1/v_J}, L^\infty)$-norm of the \twd \,has been calculated by using a window function $\Phi_{\t}$ depending on  $\t$. The next goal is to find an upper bound of this norm  independent of $\t$. We will need the following result. 
\begin{lemma}
\label{propnorm} Consider  $\Phi(x, \xi)= e^{-\pi(x^2+\xi^2)}$, $\phas\in\rdd$, and  $\Phi_{\t} = W_{\t}(\varphi, \varphi)$,   where $\varphi(t)=e^{-\pi t^2}$, $t \in\rd$.  Then, for $v_J$ in \eqref{vJ},
there exists a constant $C>0$ such that 
\begin{equation}
\label{normL1WG}
\|V_{\Phi}\Phi_{\t}\|_{L^1_{1\otimes v_J}} \leq C,\quad \forall \t\in [0,1].
\end{equation}
Consequently, 
\begin{equation}
\label{normM1WG}
\|\Phi_{\t}\|_{M^1_{1\otimes v_J}} \leq C, \quad \forall \t\in [0,1].
\end{equation}
\end{lemma}
\begin{proof}
	Using Lemma \ref{tWDgaussian} and formula \eqref{STFTGauss}, with $z=(z_1, z_2), \zeta=(\zeta_1, \zeta_2)\in\rdd$,   we compute
\begin{align*}
|V_{\Phi}\Phi_{\t}|(z, \zeta) &= \frac{1}{(2\t^2-2\t+1)^{d/2}}\frac{(2\t^2-2\t+1)^{d/2}}{(2\t^2-2\t+5)^{d/2}}  \\
& \qquad\times e^{-\pi\frac{\tfrac{3}{2\t^2-2\t+1}(z_1^2+z_2^2)+ \tfrac{2\t^2-2\t+2}{2\t^2-2\t+1}(\zeta_1^2+\zeta_2^2)-2\tfrac{2\t-1}{2\t^2-2\t+1}(z_1\zeta_2+z_2\zeta_1)}{\tfrac{2\t^2-2\t+5}{2\t^2-2\t+1}}}\\
&= \frac{1}{(2\t^2-2\t+5)^{d/2}} e^{-\pi \frac{3(z_1^2+z_2^2) + (2\t^2-2\t+2)(\zeta_1^2+\zeta_2^2)+(2-4\t)(z_1\zeta_2+z_2\zeta_1)}{2\t^2-2\t+5}}.
\end{align*}
Observing that
\begin{equation*}
\frac{1}{(2\t^2-2\t+5)^{d/2}} \leq \max_{\t\in(0,1)}\frac{1}{(2\t^2-2\t+5)^{d/2}} = \biggl(\frac{2}{9}\biggr)^{d/2},
\end{equation*}
%\begin{equation*}
%\|V_{\Phi}\Phi_{\t}\|_{L^1_z(L^1_{\zeta, v_J})} = \frac{1}{(2\t^2-2\t+5)^{d/2}} \intrdd \intrdd e^{-\pi \frac{3(z_1^2+z_2^2) + (2\t^2-2\t+2)(\zeta_1^2+\zeta_2^2)+(4\t-2)(z_1\zeta_2+z_2\zeta_1)}{2\t^2-2\t+5}} v_J(\zeta) d\zeta_1d\zeta_2 dz_1dz_2.
%\end{equation*}
by Lemma \ref{Lemgrowei}, we have,
\begin{align*}
\|V_{\Phi}\Phi_{\t}\|_{L^1_{1\otimes v_J}}&\leq \biggl(\frac{2}{9}\biggr)^{d/2} \\
&\quad\times\intrdd \intrdd e^{-\pi \frac{3(z_1^2+z_2^2) + (2\t^2-2\t+2)(\zeta_1^2+\zeta_2^2)+(2-4\t)(z_1\zeta_2+z_2\zeta_1)}{2\t^2-2\t+5}} v_J(\zeta) d\zeta_1d\zeta_2dz_1 dz_2\\
&\leq C \intrdd \intrdd e^{-\pi \frac{3(z_1^2+z_2^2) + (2\t^2-2\t+2)(\zeta_1^2+\zeta_2^2)+(2-4\t)(z_1\zeta_2+z_2\zeta_1)}{2\t^2-2\t+5}} e^{a|J\zeta|} d\zeta_1d\zeta_2 dz_1dz_2\\
&=C\intrdd e^{-\pi\frac{3(z_1^2+z_2^2)}{2\t^2-2\t+5}}I_1dz_1dz_2,
\end{align*}
where 
\begin{equation*}
I_1 :=\intrdd e^{-\pi \frac{(2\t^2-2\t+2)(\zeta_1^2+\zeta_2^2)+(2-4\t)(z_1\zeta_2+z_2\zeta_1)}{2\t^2-2\t+5}} e^{a|J\zeta|}d\zeta_1d\zeta_2.
\end{equation*}
The integral $I_1$ can be computed as follows
\begin{align*}
 I_1 &=\intrdd e^{-\pi \frac{ (2\t^2-2\t+2)(\zeta_1^2+\zeta_2^2)+(2-4\t)(z_1\zeta_2+z_2\zeta_1)}{2\t^2-2\t+5}} e^{a|J\zeta|} d\zeta_1d\zeta_2 \\
 &\leq \intrdd e^{-\pi \frac{ (2\t^2-2\t+2)(\zeta_1^2+\zeta_2^2)+(2-4\t)(z_1\zeta_2+z_2\zeta_1)}{2\t^2-2\t+5}} e^{a(|\zeta_1|+|\zeta_2|)} d\zeta_1d\zeta_2\\
 &= \biggl(\intrd e^{-\pi \frac{ (2\t^2-2\t+2)\zeta_1^2+(2-4\t)z_2\zeta_1}{2\t^2-2\t+5}} e^{a|\zeta_1|}d\zeta_1\biggr)\biggl(\intrd e^{-\pi \frac{ (2\t^2-2\t+2)\zeta_2^2+(2-4\t)z_1\zeta_2}{2\t^2-2\t+5}} e^{a|\zeta_2|}d\zeta_2\biggr).
\end{align*}
We calculate the integral with respect to the variable $\zeta_1$ (the other integral is analogous):
\begin{align*}
\intrd e^{-\pi \frac{ (2\t^2-2\t+2)\zeta_1^2+(2-4\t)z_2\zeta_1}{2\t^2-2\t+5}} e^{a|\zeta_1|}d\zeta_1&= \intrd e^{-\pi \frac{ (2\t^2-2\t+2)\zeta_1^2+(2-4\t)z_2\zeta_1+\tfrac{(1-2\t)^2z^2_2}{2\t^2-2\t+2}-\tfrac{(1-2\t)^2z^2_2}{2\t^2-2\t+2}}{2\t^2-2\t+5}} e^{a|\zeta_1|}d\zeta_1\\
&=e^{\pi\frac{(1-2\t)^2z_2^2}{(2\t^2-2\t+2)(2\t^2-2\t+5)}}\\
&\qquad\times\intrd e^{-\pi\frac{(\sqrt{2\t^2-2\t+2}\zeta_1+\tfrac{1-2\t}{\sqrt{2\t^2-2\t+2}}z_2)^2}{2\t^2-2\t+5}}e^{a|\zeta_1|}d\zeta_1 \\
&=e^{\pi\frac{(1-2\t)^2z_2^2}{(2\t^2-2\t+2)(2\t^2-2\t+5)}}\underbrace{\intrd e^{-\pi\frac{((2\t^2-2\t+2)\zeta_1+(1-2\t)z_2)^2}{(2\t^2-2\t+5)(2\t^2-2\t+2)}}e^{a|\zeta_1|}d\zeta_1}_{:= I_3}.
\end{align*} 
In $I_3$ we perform the following change of variables $$ (2\t^2-2\t+2)\zeta_1+(1-2\t)z_2= \eta_1,$$
so that, 
\begin{align*}
I_3 &= \frac{1}{(2\t^2-2\t+2)^d}\intrd e^{-\pi\frac{\eta_1^2}{(2\t^2-2\t+5)(2\t^2-2\t+2)}}e^{\tfrac{a}{2\t^2-2\t+2}|\eta_1-(1-2\t)z_2|}d\eta_1\\
&\leq C_1^de^{\tfrac{a|1-2\t|}{2\t^2-2\t+2}|z_2|}\intrd e^{-\pi C_2\eta_1^2}e^{aC_1|\eta_1|}d\eta_1,\\
\end{align*}
where 
\begin{equation*}
C_1=\max_{\t\in[0,1]}\frac{1}{(2\t^2-2\t+2)} = \frac{2}{3},\quad C_2 = \min_{\t\in[0,1]} \frac{1}{(2\t^2-2\t+5)(2\t^2-2\t+2)} = \frac{1}{10}.
\end{equation*}
Using 
$
\lim_{|\eta_1|\to \infty} e^{-\pi\frac{C_1}{2}\eta_1^2}e^{aC_2|\eta_1|} =0, 
$
for every $\epsilon>0$ there exists $R>0$ such that $e^{-\pi\frac{C_1}{2}\eta_1^2}e^{aC_2|\eta_1|}\leq \epsilon$, for all $|\eta_1|$ with $|\eta_1|>R.$ Hence
\begin{align*}
I_3 &\leq C_1^de^{\tfrac{a|1-2\t|}{2\t^2-2\t+2}|z_2|}\intrd e^{-\pi C_2\eta_1^2}e^{aC_1|\eta_1|}d\eta_1\\
&= C_1^de^{\tfrac{a|1-2\t|}{2\t^2-2\t+2}|z_2|}\left(\int_{\{\eta_1\in\rd:|\eta_1|\leq R\}} e^{-\pi C_2\eta_1^2}e^{aC_1|\eta_1|}d\eta_1 +\int_{\{\eta_1\in\rd:|\eta_1|> R\}} e^{-\pi C_2\eta_1^2}e^{aC_1|\eta_1|}d\eta_1 \right)\\
&\leq C_1^de^{\tfrac{a|1-2\t|}{2\t^2-2\t+2}|z_2|}\left(e^{aC_1R}\intrd e^{-\pi C_2\eta_1^2}d\eta_1 + \epsilon\intrd e^{-\pi\frac{C_2}{2}\eta_1^2}d\eta_1\right) =\tilde{C} e^{\tfrac{a|1-2\t|}{2\t^2-2\t+2}|z_2|}<\infty,
\end{align*}
where $\tilde{C}$ is a constant independent of $\tau.$ In conclusion, the integral $I_1$ can be majorized as
\begin{align*}
I_1\leq 2\tilde{C}  e^{\pi\frac{(1-2\t)^2z_2^2}{(2\t^2-2\t+2)(2\t^2-2\t+5)}+\tfrac{a|1-2\t|}{2\t^2-2\t+2}|z_2|}e^{\pi\frac{(1-2\t)^2z_1^2}{(2\t^2-2\t+2)(2\t^2-2\t+5)}+\tfrac{a|1-2\t|}{2\t^2-2\t+2}|z_1|}.
\end{align*}
Thus, there exists a constant $M_1>0$ independent of $\t$ such that
\begin{align*}
\|V_{\Phi}\Phi_{\t}\|_{L^1_{1\otimes v_J}}&\leq  M_1\intrd e^{-\pi\tfrac{3z_1^2}{2\t^2-2\t+5}}e^{\pi\tfrac{(1-2\t)^2z_1^2}{(2\t^2-2\t+2)(2\t^2-2\t+5)}+\tfrac{a|1-2\t|}{2\t^2-2\t+2}|z_1|}dz_1 \\
&\hspace{2cm}\times\intrd e^{-\pi\tfrac{3z_2^2}{2\t^2-2\t+5}}e^{\pi\tfrac{(1-2\t)^2z_2^2}{(2\t^2-2\t+2)(2\t^2-2\t+5)}+\tfrac{a|1-2\t|}{2\t^2-2\t+2}|z_2|}dz_2\\
&=2M_1\intrd e^{-\pi\tfrac{3z_1^2}{2\t^2-2\t+5}}e^{\pi\tfrac{(1-2\t)^2z_1^2}{(2\t^2-2\t+2)(2\t^2-2\t+5)}+\tfrac{a|1-2\t|}{2\t^2-2\t+2}|z_1|}dz_1.
\end{align*}
The integral with respect the variable $z_1$ is computed analogously to the one for $\zeta_1$ above.  The estimate \eqref{normM1WG} follows by 
\begin{equation*}
\|\Phi_{\t}\|_{M^1_{1\otimes v_J}}\asymp \|V_{\Phi}\Phi_{\t}\|_{L^1_{1\otimes v_J}} \leq C,
\end{equation*}
as desired.
\end{proof}

%Thanks to Lemma \ref{Wigwienerlemma} and Proposition \ref{propnorm}, we are able to demostrate the following result.
\begin{proposition}
\label{prop1}
Under the assumptions  of Lemma \ref{Wigwienerlemma}, there exists a  constant $C>0$ independent of $\tau$ such that 
\begin{equation}
\label{estimate1}
\|W_{\t}(g,f)\|_{W(\Ft L^1_{1/v_J}, L^{\infty})}\leq C \alpha_{(p_1,p_2)}(\tau)\|f\|_{M^{p_1,p_2}_m}\|g\|_{M^{p_1',p'_2}_{1/m}},\quad \t\in(0,1).
\end{equation}
\end{proposition}
\begin{proof}
Changing window in the computation of the STFT as in \eqref{chwindow}, using Lemmas \ref{propconvL1}, \ref{propnorm} and Moyal's formula \eqref{C1MoyaltauWigner}, we have
\begin{align*}
\|V_{\Phi}W_{\t}(g,f)\|_{L^{\infty}_z(L^1_{\zeta, 1/v_J})} &\leq \frac{1}{|\la\Phi_{\t}, \Phi_{\t}\ra|}\||V_{\Phi_{\t}}W_{\t}(g,f)|\ast |V_{\Phi}\Phi_{\t}|\|_{L^{\infty}_z(L^1_{\zeta, 1/v_J})}\\
&\leq \frac{1}{\|\varphi\|^2\|\varphi\|^2} \|V_{\Phi_{\t}}W_{\t}(g,f)\|_{L^{\infty}_z(L^1_{\zeta, 1/v_J})} \|V_{\Phi}\Phi_{\t}\|_{L^1_{1\otimes v_{J}}} \\
&\leq C \alpha_{(p_1,p_2)}(\tau)\|f\|_{M^{p_1,p_2}_m}\|g\|_{M^{p_1',p'_2}_{1/m}}.
\end{align*}
This completes the proof.
\end{proof}

Repeating the pattern of Lemma \ref{Wigwienerlemma} and Proposition \ref{prop1} in the Wiener amalgam space  $W(\Ft L^2_{1/v_J}, L^2)(\rdd)$, we can state the following.
\begin{proposition}
\label{tWDW2}
Let $m\in \mathcal{M}_v(\rdd)$, $f\in M^2_m(\rd)$ and $g\in M^2_{1/m}(\rd)$. For $\t\in(0,1)$, the \twd \, $W_{\t}(g,f)$ is in $W(\Ft L^2_{1/v_J}, L^2)(\rdd)$, with the uniform estimate
\begin{equation}
\label{estimate2}
\|W_{\t}(g,f)\|_{W(\Ft L^2_{1/v_J}, L^2)} \leq C \|f\|_{ M^2_m}\|g\|_{ M^2_{1/m}},
\end{equation}
where the positive constant $C$ is independent of $\t$.
\end{proposition}
\begin{proof}
\emph{First Step}. We use Lemma \ref{STFT-tau-wigner},  Young's Inequality $L^1\ast L^1\subset L^1$ and the change of variables $\mathcal{B}_\tau\eta\rightarrow \eta$, to compute
\begin{align*}
   &\|W_\tau(g,f)\|_{W(\Ft L^2_{1/v_J}, L^2)}\\ &\asymp  \left( \intrdd  \intrdd  |V_{\f_1} g(z+\sqrt{\tau(1-\tau)}\A^T\zeta)|^2|V_{\f_2} f(z+\sqrt{\tau(1-\tau)}\A\zeta)|^2\frac{1}{v^2(J\zeta)}d\zeta dz\right)^\frac 12\\
   &\leq C\frac{1}{[\tau(1-\tau)]^d} \left( \intrdd  \intrdd  |V_{\f_1} g(z+\eta -\mathcal{B}_\tau\eta)|^2|V_{\f_2} f(z+\eta)|^2\frac{m^2(z+\eta)}{m^2(z+\eta-\mathcal{B}_{\tau}\eta)}d\eta dz\right)^\frac 12\\
    &= C \frac{1}{[\tau(1-\tau)]^d} \left(\intrdd (|V_{\f_2} f|^2 m^2)\ast (|V_{\f_1} g|^2 \frac1{m^2})(\mathcal{B}_{\t}\eta)\, d\eta\right)^\frac12\\
    & \lesssim \| |V_{\f} f|^2 m^2\|_1 \||V_{\f} g|^2 \frac1{m^2}\|_1\\
    &\lesssim \|f\|_{M^{2}_m}\|g\|_{M^{2}_{1/m}}.
  \end{align*}

\noindent
\emph{Second Step}. Consider now $\Phi\in\cS(\rdd)$. Then the same pattern as in the proof of Proposition \ref{prop1}, with Lemma  \ref{propconvL1} replaced by Lemma \ref{propconL2}, gives the uniform estimate \eqref{estimate2}.
\end{proof}

The previous issue can be rephrased in terms of modulation spaces as follows (cf. \eqref{W-M}).
\begin{corollary}
For $\tau\in (0,1)$, $m\in\mathcal{M}_v(\rdd)$, $f\in M^2_m(\rd)$, $g\in M^2_{1/m}(\rd)$, the \twd \, belongs to $M^2_{1/v_J\otimes 1}(\rdd)$ with
$$\|W_\tau(g,f)\|_{M^2_{1/v_J\otimes 1}}\leq C \|f\|_{M^2_m}\|g\|_ {M^2_{1/m}},
$$
with $C>0$ independent of $\tau$.
\end{corollary}

\section{Main result}
This section is devoted to the proof of Theorem \ref{mainthm}. We will start with two preliminary results about $\t$-pseudodifferential operators acting on modulation spaces and having symbols in $W(\Ft L^1_{1/v_J}, L^{\infty})(\rdd)$ and $W(\Ft L^2_{1/v_J}, L^2)(\rdd)$, respectively. Then, by means of complex interpolation between Wiener amalgam spaces, we shall reach our goal.
\begin{proposition}\label{Prop1} 
Suppose that $m \in \mathcal{M}_{v}(\rdd)$ and consider a symbol function $a \in W(\cF L^\infty_{v_J},L^1)(\rdd)$. Then for every $\tau\in (0,1)$, the $\tau$-pseudodifferential operator $\operatorname*{Op}\nolimits_{\tau}(a)$ is bounded on $M^{p_1,p_2}_m(\rd)$, for every $1\leq p_1,p_2\leq\infty$, with 
\begin{equation}\label{stimae0}
\|\operatorname*{Op}\nolimits_{\tau}(a)f\|_{{M}_m^{p_1,p_2}} \leq C \alpha_{(p_1,p_2)}(\tau)
\|a\|_{W(\cF L^\infty_{v_J},L^1)}\|f\|_{{M}_m^{p_1,p_2}}
\end{equation}
($C>0$ does not depend on $\tau$).
\end{proposition}
\begin{proof} For every $f\in {M}_m^{p_1,p_2}(\rd)$ and $g\in {M}_{1/m}^{p'_1,p'_2}(\rd)$, we can write
\begin{equation*}
| \la \operatorname*{Op}\nolimits_{\tau}(a) f, g\ra |=|\la a, W_\tau(g,f)\ra|\leq \|V_{\Phi} a\|_{L^1_z(L^\infty_{{v_J},\z})}\|V_{\Phi} W_\tau(g,f)\|_{L^\infty_z(L^1_{{{1/v_J}},\z})}.
\end{equation*}
Observing that
\begin{equation*}
 \|W_\tau(g,f)\|_{W(\cF L^1_{1/v_J},L^\infty)}\asymp \|V_{\Phi} W_\tau(g,f)\|_{L^\infty_z(L^1_{{1/v_J},\z})}
\end{equation*}
and using  Proposition \ref{prop1}, we conclude the proof.
\end{proof}

\begin{proposition}\label{Prop2} 
Let $m \in \mathcal{M}_{v}(\rdd)$, $a\in W(\cF L^2_{v_J},L^2)(\rdd)$ and $\t\in (0,1)$. Then the operator $\operatorname*{Op}\nolimits_{\tau}(a)$ is bounded on $M^{2}_m$ with
\begin{equation}\label{stimae02}
\|\operatorname*{Op}\nolimits_{\tau}(a)f\|_{{M}_m^{2}} \leq C 
\|a\|_{W(\cF L^2_{v_J},L^2)}\|f\|_{{M}_m^{2}},
\end{equation}
where the constant $C>0$ is independent of  $\t$. 
\end{proposition}
\begin{proof}
The proof is similar to the one of Proposition \ref{Prop1}, where Proposition \ref{prop1} is replaced by \ref{tWDW2}.
\end{proof}
\begin{comment}
\begin{remark} (i) Observe that by \eqref{W-M}, $W (\cF L^2_{v_J},L^2)=\cF M^2_{v_J\otimes 1}$ and a straightforward modification of \cite[Theorem 11.3.5 (c)]{G}
gives
$$ \cF M^2_{v_J\otimes 1}= M^2_{1\otimes v_{J^{-1}}}= M^2_{1\otimes v_{J}}
$$
since by assumption $v(-z)=v(z)$.\par
(ii)  Since $v(-z)=v(z)$, the weight $v_J$ is even and the conclusion of the previous step (i)  also follows by \cite[Theorem 6]{F3}, in the case $p=2$.\par
(iii) Using (i) or (ii) we derive that the Wiener amalgam space $W(\cF L^2_{v_J},L^2)$ coincides with the modulation space  $M^2_{1\otimes v_J}$. Then the conclusion of Proposition \ref{Prop2} also follows from \cite[Theorem 4.3]{T1}.
\end{remark}
\end{comment}

Propositions \ref{Prop1} and \ref{Prop2} are the main ingredients in the proof of Theorem \ref{mainthm}, which generalizes \cite[Theorem 3.1]{DT} in the case of $\t$-pseudodifferential operators.

%\begin{theorem}\label{contwiener} Assume that $1 \leq p,q,r_1,r_2 \leq \infty$ satisfy
%\begin{equation}\label{e1}
%q\leq p'
%\end{equation}
%and
%\begin{equation}\label{e2}
%\max \{r_1,r_2,r_1',r_2'\}\leq p.\end{equation}
%Consider $m \in \mathcal{M}_{v}(\rdd)$. Then for any $\tau\in (0,1)$, $\tau$-operator $\operatorname*{Op}\nolimits_{\tau}(a)$ having symbol $a\in W(\cF L^p_{v_J}, L^q)(\rdd)$, from $\cS(\rd)$ to $\cS'(\rd)$, extends uniquely
%to a bounded operator on $M^{r_1,r_2}_m(\R^d)$, with the estimate
%\begin{equation}\label{stimae1}
%\|\operatorname*{Op}\nolimits_{\tau}(a)f\|_{M_m^{r_1,r_2}} \leq C \frac{1}{\tau^{(1-\tfrac{1}{r_1}+\tfrac{1}{r_2})^d}(1-\tau)^{(1+\tfrac{1}{r_1}-\tfrac{1}{r_2})^d}}
%\|a\|_{W(\cF L^p_{v_J},L^q)}\|f\|_{M_m^{r_1,r_2}}.
%\end{equation}
%\end{theorem}

\begin{theorem}
	\label{mainthm}
	Suppose that $1 \leq p,q,r_1,r_2 \leq \infty$ satisfy
	\begin{equation}
	\label{e1M}
	q\leq p'
	\end{equation}
	and
	\begin{equation}
	\label{e2M}
	\max \{r_1,r_2,r_1',r_2'\}\leq p.\end{equation}
	Let $m \in \mathcal{M}_{v}(\Renn)$ and $a\in W(\cF L^p_{v_J}, L^q)(\rdd)$ . For $\tau\in(0,1)$, every \tpsdo\ $\tauop$ is a bounded operator on $M^{r_1,r_2}_m(\R^d)$. Moreover, there exists a  constant $C>0$ independent of $\tau$ such that
	\begin{equation}\label{stimae1}
	\|\operatorname*{Op}\nolimits_{\tau}(a)f\|_{M_m^{r_1,r_2}} \leq C \alpha_{(r_1,r_2)}(\tau)
	\|a\|_{W(\cF L^p_{v_J},L^q)}\|f\|_{M_m^{r_1,r_2}},\quad  \tau \in (0,1).
	\end{equation}
\end{theorem}
\begin{proof} The key tool is the complex interpolation between Wiener amalgam and modulation spaces. We regard $\tauopnosy$ as the bilinear map $(a, f)\mapsto \tauop f$. Proposition \ref{Prop1} and Proposition \ref{Prop2} give the continuity of the $\t$-pseudodifferential operator $\tauopnosy$  on the following function spaces
\begin{align*}
W(\cF L^{\infty}_{v_J}, L^1)(\rdd) \times M^{p_1,p_2}_m(\rd) &\rightarrow M^{p_1,p_2}_m(\rd),\\
W(\cF L^{2}_{v_J}, L^2)(\rdd)\times M^2_m(\rd) & \rightarrow M^2_m(\rd),
\end{align*} 
for $1\leq p_1, p_2 \leq \infty$.  Using the complex interpolation between Wiener amalgam  and modulation spaces \cite{F2}, for $\theta\in [0,1]$, we have
$$ [W(\cF L^{\infty}_{v_J}, L^1), W(\cF L^{2}_{v_J}, L^2)]_{\theta}=W(\cF L^{p}_{v_J}, L^{p'}),
$$
with $2\leq p\leq\infty$, and 
$ [M^{p_1,p_2}_m, M^2_m]_\theta=M^{r_1,r_2}_m$,
with
\begin{equation}
\label{indeces1}
\frac1{r_1}=\frac {1-\theta}{p_1}+\frac\theta 2=\frac {1-\theta}{p_1}+\frac1p
\end{equation}
 and
\begin{equation}
\label{indeces2}
\frac1{r_2}=\frac {1-\theta}{p_2}+\frac\theta 2=\frac {1-\theta}{p_2}+\frac1p
\end{equation}
so that $r_1,r_2\leq p$. Similarly, we obtain $r_1',r_2'\leq p$, and thus the relation \eqref{e2M}.
Due to inclusion relations for Wiener amalgam spaces, we relax the assumptions on symbols, so that the symbol $a$ may belong to $W(\cF L^{p}_{v_J}, L^{q})(\rdd)$, with $q\leq p'$, which gives \eqref{e1M}. Finally, the norm is provided by
\begin{align*}
\|\operatorname*{Op}\nolimits_{\tau}\|_{\mathcal{B}(W(\cF L^{p}_{v_J}, L^q) \times M^{r_1,r_2}_m, M^{r_1,r_2}_m)}&\leq \|\operatorname*{Op}\nolimits_{\tau}\|_{\mathcal{B}(W(\cF L^{\infty}_{v_J}, L^1) \times M^{p_1,p_2}_m, M^{p_1,p_2}_m)}^{1-\theta}\\
&\hspace{2cm}\times \|\operatorname*{Op}\nolimits_{\tau}\|_{\mathcal{B}(W(\cF L^{2}_{v_J}, L^2) \times M^2_m, M^2_m)}^{\theta}\\
&\leq C\frac{1}{\t^{d(1-\theta)\left(1-\tfrac{1}{p_1}+\tfrac{1}{p_2}\right)}(1-\t)^{d(1-\theta)\left(1+\tfrac{1}{p_1}-\tfrac{1}{p_2}\right)}}\\
&\leq C \frac{1}{\t^{d\left(1-\tfrac{1}{p_1}+\tfrac{1}{p_2}\right)}(1-\t)^{d\left(1+\tfrac{1}{p_1}-\tfrac{1}{p_2}\right)}},
\end{align*}
since $1-\theta \leq 1$. This concludes the proof.
%Thanks to relations \eqref{indeces1} and \eqref{indeces2}, we obtain the desided formula \eqref{stimae1}, which concludes our proof.
\end{proof}
We finally consider the end-points $\tau=0$ and $\tau=1$, for which the boundedness results stated above do not hold in general. We remark that the modulation space $M^2(\rd)$ is simply the Lebesgue space $L^2(\rd).$ The following example generalizes a $1$-dimensional example exhibited by Boulkhemair in \cite{Boul}.
\begin{proposition}\label{controesempi}
There exists a  symbol $a\in W(\Ft L^{\infty}, L^1)(\rdd)$ such that the corresponding  Kohn-Nirenberg  $\Opkn(a)$ and anti-Kohn-Nirenberg $\Opknc(a)$ operators are not bounded on $L^2(\rd)$.
\end{proposition}
 
\begin{proof}
Consider  the symbol function
\begin{equation}
\label{symbol}
a(x_1,\dots, x_d,\xi_1,\dots, \xi_d)= x_1^{-1/2}\dots x_d^{-1/2}\chi_{(0,1]}(x_1)\dots \chi_{(0,1]}(x_d)e^{-\pi\xi^2},
\end{equation}
with $\xi^2=\xi_1^2+\dots+\xi_d^2$.
An easy computation shows that $a\in L^1(\rdd)=W(L^1, L^1)(\rdd)\subset W(\mathcal{F} L^{\infty}, L^1)(\rdd)$. Let us show that the Kohn-Niremberg $\Opkn(a)$ is unbounded on $L^2(\rd)$. Consider the Gaussian function $f(t)= e^{\pi t^2}\in L^2(\rd)$, then $\Opkn(a)f\notin L^2(\rd)$. Indeed, by a tensor product argument, we reduce to compute the following one-dimensional integral:
$$
\int_{\mathbb{R}}e^{2\pi ix\xi}x^{-1/2}\chi_{(0,1]}(x)e^{-\pi\xi^2}e^{-\pi\xi^2}d\xi = \frac{1}{\sqrt{2}}x^{-1/2}\chi_{(0,1]}(x)e^{-\pi\tfrac{x^2}{2}},
$$
whose result is a function that does not belong to $L^2(\R)$.\par 
To prove that the anti-Kohn-Nirenberg operator $\operatorname*{Op}\nolimits_{1}(a)$, where $a$ is defined in \eqref{symbol}, is unbounded on $L^2 (\rd)$, it is sufficient to observe that its adjoint operator is the Kohn-Niremberg one: $(\operatorname*{Op}\nolimits_{1}(a))^*= \Opkn(a)$, as detailed below:
\begin{equation*}
\langle \operatorname*{Op}\nolimits_{1}(a) f, g\rangle = \langle a, \cR^*(g,f)\rangle = \langle a, \overline{\cR(f,g)}\rangle = \langle \cR(f,g), \overline{a} \rangle = \langle f,\operatorname*{Op}\nolimits_{0}(a) g\rangle. 
\end{equation*}
This proves our claim.
\end{proof}

\section{Remarks on boundedness results for symbols in modulation spaces}
We address this section to study the boundedness results for $\t$-pseudodifferential operators with symbols in weighted modulation spaces. Recall (cf. \cite{H} and  \cite[Remark 1.5]{T}) that for every choice $\tau_1,\tau_2\in [0,1]$, $a_1,a_2\in \cS'(\rdd)$,
\begin{equation}\label{linktausymb1}
 \operatorname*{Op}\nolimits_{\tau_1}(a_1)= \operatorname*{Op}\nolimits_{\tau_2}(a_2)\, \Leftrightarrow\,\widehat{a_2}(\xi_1,\xi_2)=e^{-2\pi i(\tau_2-\tau_1)\xi_1\xi_2}\widehat{a_1}(\xi_1,\xi_2).
\end{equation}
 For $t>0$ define $H_t(x,\xi)=e^{2\pi i t x \xi}$ and observe that
\begin{equation}\label{ftchirp}\cF H_t (\zeta_1,\zeta_2)=\frac 1{t^{d}} e^{-2\pi i\frac1t \zeta_1\zeta_2 }.\end{equation}
 So, for $\tau_1\not=\tau_2$, by \eqref{ftchirp},
\begin{equation}\label{linktausymb2}
 a_2 \phas =\frac1{|\tau_1-\tau_2|^d} e^{2\pi i(\tau_2-\tau_1)\Psi}\ast a_1 \phas,
\end{equation}
where $\Psi\phas =x\o$. Toft in  \cite[Proposition 1.2 (5)]{T} proved that  the mapping $a\mapsto T_\Phi a= e^{2\pi i\Phi}\ast a$ is a homeomorphism on $M^{p,q}(\rdd)$, $1\leq p,q\leq \infty$. This implies that results for Weyl operators with symbols in modulation spaces are still true for any $\t$-operator.  The main goal of this section is to show uniform estimates  for $\t$-operators with symbol in weighted modulation spaces. Following the pattern of the previous section, we first compute the norm of the \twd\,  in weighted modulation spaces.

The next Proposition extends the sufficient conditions of \cite[Theorem 1.1]{CN} in the case of $\t$-Wigner distributions.
\begin{prop}
\label{prop:Wmodspa}
Assume that $p_1, p_2, q_1, q_2, p, q \in [0,1]$ satisfy 
\begin{equation}
\label{indices5}
p_i, q_i\leq q,\qquad i=1, 2, 
\end{equation}
and 
\begin{equation}
\label{indices6}
\frac{1}{p_1} + \frac{1}{p_2}\geq \frac{1}{p}+\frac{1}{q}, \qquad \frac{1}{q_1}+\frac{1}{q_2}\geq \frac{1}{p}+\frac{1}{q}.
\end{equation}
Consider $m\in \mathcal{M}_v$, $f\in M^{p_1, q_1}_{m}(\rd)$ and $g\in M^{p_2, q_2}_{1/m}(\rd)$. Then, for any $\t\in [0, 1]$, $W_{\t}(g, f)\in M^{p, q}_{1\otimes 1/v_J}(\rdd)$. Furthermore, there exists a constant $C>0$, independent of $\t\in [0,1]$, such that
$$
\|W_{\t}(g, f)\|_{ M^{p, q}_{1\otimes 1/v_J}} \leq C\|f\|_{ M^{p_1, q_1}_{m}}\|g\|_{ M^{p_2, q_2}_{1/m}},\quad \forall \tau\in [0,1]. 
$$
\end{prop}
\begin{proof}
We separate the proof in three cases: $\t\in (0,1)$, $\t=0$ and $\t=1$. \\ \textbf{Case $\t\in (0, 1).$} \textit{Assume $p\leq q<\infty.$} Making the change of variables $z+\sqrt{\t(1-\t)}\mathcal{A}_{\t}\zeta = y$ and using item (iv) of Lemma \ref{decmatrA}, 
the integral with respect the variable $z$ becomes
\begin{align*}
\|& W_{\t}(g, f)\|_{ M^{p, q}_{1\otimes 1/v_J}}	\\
& = \biggl(\intrdd \biggl(\intrdd |V_{\varphi_1}g(z+\sqrt{\t(1-\t)}\mathcal{A}_{\t}^T\zeta)| ^p |V_{\varphi_2}f(z+\sqrt{\t(1-\t)}\mathcal{A}_{\t}\zeta)| ^p dz \biggr)^{\frac{q}{p}}\frac{1}{v^q(J\zeta)}d\zeta\biggr)^{\frac{1}{q}}\\
&= \biggl(\intrdd \biggl(\intrdd |V_{\varphi_1}g(y-J\zeta)| ^p |V_{\varphi_2}f(y)| ^p \frac{1}{v^p(J\zeta)} dy \biggr)^{\frac{q}{p}}d\zeta\biggr)^{\frac{1}{q}}\\
&\leq C\biggl(\intrdd \biggl(\intrdd |V_{\varphi_1}g(y-J\zeta)| ^p |V_{\varphi_2}f(y)| ^p \frac{m^p(y)}{m^p(y-J\zeta)} dy \biggr)^{\frac{q}{p}}d\zeta\biggr)^{\frac{1}{q}}\\
&=C\biggl(\intrdd \biggl( (|\mathcal{I} V_{\varphi_1}g| ^p\frac{1}{m^p}) \ast (|V_{\varphi_2}f| ^p m^p)(J\zeta)  \biggr)^{\frac{q}{p}}d\zeta\biggr)^{\frac{1}{q}}\\
&= C\||(\mathcal{I} V_{\varphi_1}g| ^p \frac{1}{m^p})\ast (|V_{\varphi_2}f| ^pm^p)\|_{L^{q/p}}^{1/p},
\end{align*}
where $\mathcal{I}$ is the reflection operator. The rest goes exactly as in the proof of 
 Theorem 3.1 in \cite{CN}, obtaining the estimate
$$
\| W_{\t}(g, f)\|_{M^{p, q}_{1\otimes 1/v_J}} ^p \lesssim\|f\|_{M^{p_1, q_1}_m} \|f\|_{M^{p_2, q_2}_{1/m}}.
$$
Using Lemma \ref{propnorm},  there exists a positive constant $C$ independent of $\t\in (0,1)$ such that  
\begin{align*}
\| W_{\t}(g, f)\|_{M^{p, q}_{1\otimes 1/v_J}} &\leq \frac{1}{|\langle \Phi_{\t}, \Phi_{\t}\rangle|} \| V_{\Phi_{\t}}W_{\t}(g, f)\ast V_{\Phi}\Phi_{\t} \| _{L^{p, q}_{1\otimes 1/v_J}}
%&= \| V_{\Phi_{\t}}W_{\t}(g, f) \|_{L^{p, q}_{1\otimes 1/v_J}} \|V_{\Phi}\Phi_{\t}\|_{L^1_{1\otimes v}}\\
%&=\| W_{\t}(g, f) \|_{M^{p, q}_{1\otimes 1/v_J}} \|\Phi_{\t}\|_{M^1_{1\otimes v}}\\
\leq C\|f\|_{M^{p_1, q_1}_{m}}\|g\|_{M^{p_2, q_2}_{1/m}},
\end{align*}
concluding the proof for $\tau\in (0,1)$, $p\leq q<\infty$.\\
\textit{Assume $p=q=\infty$.} We have\par 
\begin{align*}
\| W_{\t}(g, f)\|_{M^{\infty}_{1\otimes 1/v_J}} \!\!\!&= \sup_{z,\zeta\in\rdd}|V_{\varphi_1}g(z+\sqrt{\t(1-\t)}\mathcal{A}^T_{\t}\zeta)||V_{\varphi_2}f(z+\sqrt{\t(1-\t)}\mathcal{A}_{\t}\zeta)|\frac{1}{v(J\zeta)}\\
&= \sup_{\zeta\in\rdd}\sup_{z\in\rdd}|V_{\varphi_1}g(z-J\zeta)||V_{\varphi_2}f(z)|\frac{1}{v(J\zeta)},\\
&\leq C\sup_{\zeta\in\rdd}\sup_{z\in\rdd}|V_{\varphi_1}g(z-J\zeta)||V_{\varphi_2}f(z)|\frac{m(z)}{m(z-J\zeta)},\\
&= C\sup_{z\in\rdd}|V_{\varphi_2}f(z)|m(z)\biggl(\sup_{\zeta\in\rdd}|V_{\varphi_1}g(z-J\zeta)|\frac{1}{m(z-J\zeta)}\biggr),\\
&= C\||V_{\varphi_2}f|m\|_{L^{\infty}}\||V_{\varphi_2}g|\frac{1}{m}\|_{L^{\infty}}\\
&=\|f\|_{M^{\infty}_{1/v}}\|g\|_{M^{\infty}_{v}}\\
&\leq \|f\|_{M^{p_1, q_1}_{1/v}}\|g\|_{M^{p_2, q_2}_{v}},
\end{align*}
%where we make the change of variables $z+\sqrt{\t(1-\t)}\mathcal{A}_{\t}\zeta \mapsto z.$ Now if we use 
%$$
%\frac{1}{v(z-J\zeta)} \leq %v(z-J\zeta)\frac{1}{v(z)},
%$$
%then 
%\begin{align*}
%\|W_{\t}(g, f)\|_{M^{\infty}_{1\otimes 1/v_J}} &\leq \sup_{\zeta\in\rdd}\sup_{z\in\rdd}|V_{\varphi_1}g(z-J\zeta)||V_{\varphi_2}f(z)|v(z-J\zeta)\frac{1}{v(z)} \\
%&\leq \sup_{\zeta\in\rdd}|V_{\varphi_1}g(z-J\zeta)|v(z-J\zeta) \sup_{z\in\rdd}|V_{\varphi_2}f(z)|\frac{1}{v(z)} \\
%&= \|f\|_{M^{\infty}_{1/v}}\|g\|_{M^{\infty}_{v}}\\
%&\leq \|f\|_{M^{p_1, q_1}_{1/v}}\|g\|_{M^{p_2, q_2}_{v}},
%\end{align*}
for every $1\leq p_i, q_i \leq \infty$. The conclusion follows again by  Lemma \ref{propnorm} and  Young's inequality.\\
\textit{Assume $p>q$.} Using the inclusion relations for modulation spaces, we majorize 
$$
\| W_{\t}(g, f)\|_{M^{p, q}_{1\otimes 1/v_J}} \leq \| W_{\t}(g, f)\|_{M^{q, q}_{1\otimes 1/v_J}} \leq C \|f\|_{M^{p_1, q_1}_{m}}\|g\|_{M^{p_2, q_2}_{1/m}},
$$
for every $1\leq p_i, q_i \leq q,$ $i=1, 2$. %Here we have applied the case $p\leq q$ with $p=q$. 
\\
\textbf{Case $\t=0$.} In this case, we obtain at once a uniform estimate. Indeed, using Lemma \ref{STFT-Rihaczek},
\begin{align*}
\|W_{0}&(g,f)\|_{M^{p, q}_{1\otimes 1/v_J}}  \\&= \biggl(\intrd\biggl(\intrd |V_{\varphi_1}g(z_1, z_2+\zeta_1)|^{p}|V_{\varphi_2}f(z_1+\zeta_2, z_2)|^{p}dz_1dz_2\biggr)^{q/p}\frac{1}{v_J^{q}(\zeta)}d\zeta_1d\zeta_2\biggr)^{1/q}\\
&= \biggl(\intrd\biggl(\intrd |V_{\varphi_1}g(z_1 - \zeta_2, z_2+\zeta_1)|^{p}|V_{\varphi_2}f(z_1, z_2)|^{p}\frac{1}{v^{p}(J\zeta)}dz_1dz_2\biggr)^{q/p}d\zeta_1d\zeta_2\biggr)^{1/q}\\
&\leq C \biggl(\intrd\biggl(\intrd |V_{\varphi_1}g(z-J\zeta)|^{p}|V_{\varphi_2}f(z)|^{p}\frac{m^{p}(z)}{m^{p}(x-J\zeta)}dz\biggr)^{q/p}d\zeta\biggr)^{1/q}\\
&=C\|(|V_{\varphi_1}g|^{p}(1/m^{p}))\ast (|V_{\varphi_2}f|^{p}m^{p})\|_{L^{q/p}}.
\end{align*} 
Then we proceed as in Case $\t\in (0,1).$\\
\textbf{Case $\t=1$.} The proof is analogous to the one of {Case $\t=0$}. We are done.
\end{proof}

The boundedness results for \twd s transfer  to $\t$-pseudodifferential operators as follows.
\begin{theorem}
Let $1\leq p_1, p_2, q_1, q_2, p\leq\infty$ be indices  such that
\begin{equation}
\label{indices3}
p_1, p'_2, q_1, q'_2\leq q',
\end{equation}
and 
\begin{equation}
\label{indices4}
\frac{1}{p_1}+ \frac{1}{p'_2} \geq \frac{1}{p'}+ \frac{1}{q'}, \quad  \frac{1}{q_1}+ \frac{1}{q'_2} \geq \frac{1}{p'}+ \frac{1}{q'}.
\end{equation}
Let $m\in \mathcal{M}_{v}(\rdd)$. For every $\t\in [0,1]$, the $\t$-pseudodifferential operator $\tauop$, with symbol $a\in M^{p, q}_{1\otimes v_J}(\rdd)$, is a bounded operator from $M^{p_1, q_1}_m(\rd)$ to $M^{p_2, q_2}_m(\rd)$, with 
$$
\|\tauop f\|_{M^{p_2, q_2}_m}\leq C\|a\|_{M^{p, q}_{1\otimes v_J}}\|f\|_{M^{p_1, q_1}_m},
$$
and $C>0$ is independent of $\tau$.
\end{theorem}
\begin{proof}
If $f\in M^{p_1, q_1}_m(\rd)$ and $g\in M^{p'_2, q'_2}_{1/m}(\rd)$, then $W_{\t}(g,f)\in M^{p', q'}_{1\otimes\tfrac{1}{v_J}}(\rdd)$, by Proposition \ref{prop:Wmodspa}, provided that \eqref{indices3} and \eqref{indices4} hold. Thereby there exists a positive constant $C$ such that for any $\t\in [0, 1]$,
\begin{align*}
|\langle\tauop f, g\rangle| &= |\langle a, W_{\t}(g,f)| \\
&\leq C \|a\|_{M^{p,q}_{1\otimes v_J}} \|f\|_{M^{p_1, q_1}_m} \|g\|_{M^{p'_2, q'_2}_{1/m}},
\end{align*} 
as desired.
\end{proof}

\bibliographystyle{plain}

%\bibliography{bibliografia}

\end{document}